\documentclass[11pt]{article}
\usepackage{amsmath}
\usepackage{amssymb}
\usepackage{theorem}
\usepackage[usenames,dvipsnames]{pstricks}
\usepackage{pstricks-add}
\usepackage{pst-grad}
\usepackage{pst-plot} 
\usepackage{pst-eucl} 
\usepackage{euscript}
\usepackage{epic,eepic}
\usepackage{ulem}
\PassOptionsToPackage{normalem}{ulem}
\usepackage{graphicx}
\usepackage{epsfig}
\newcommand{\minimize}[2]{\ensuremath{\underset{\substack{{#1}}}%
{\mathrm{minimize}}\;\;#2 }}
\newcommand{\Frac}[2]{\displaystyle{\frac{#1}{#2}}} 
\newcommand{\scal}[2]{{\left\langle{{#1}\mid{#2}}\right\rangle}}

\newcommand{\menge}[2]{\big\{{#1}~\big |~{#2}\big\}} 
\newcommand{\cyc}{\ensuremath{\mathsf{cyc}}}
\newcommand{\HH}{\ensuremath{{\mathcal H}}}
\newcommand{\HHH}{\ensuremath{\boldsymbol{\mathcal H}}}

\newcommand{\emp}{\ensuremath{{\varnothing}}}

\newcommand{\RR}{\ensuremath{\mathbb{R}}}
\newcommand{\RP}{\ensuremath{\left[0,+\infty\right[}}

\newcommand{\RPP}{\ensuremath{\left]0,+\infty\right[}}

\newcommand{\NN}{\ensuremath{\mathbb N}}

\newcommand{\cart}{\ensuremath{\raisebox{-0.5mm}{\mbox{\LARGE{$\times$}}}}\!}

\newcommand{\pinf}{\ensuremath{{+\infty}}}

\newcommand{\Fix}{\ensuremath{\operatorname{Fix}}}

\def\proof{\noindent{\it Proof}. \ignorespaces}
\def\endproof{\vbox{\hrule height0.6pt\hbox{\vrule height1.3ex%
width0.6pt\hskip0.8ex\vrule width0.6pt}\hrule height0.6pt}}
\newtheorem{theorem}{Theorem}[section]

\newtheorem{corollary}[theorem]{Corollary}

\theoremstyle{plain}{\theorembodyfont{\rmfamily}%
}
\theoremstyle{plain}{\theorembodyfont{\rmfamily}%
}
\theoremstyle{plain}{\theorembodyfont{\rmfamily}%
}
\theoremstyle{plain}{\theorembodyfont{\rmfamily}%
\newtheorem{remark}[theorem]{Remark}}
\theoremstyle{plain}{\theorembodyfont{\rmfamily}%
\newtheorem{definition}[theorem]{Definition}}
\theoremstyle{plain}{\theorembodyfont{\rmfamily}%
\newtheorem{question}[theorem]{Question}}

\numberwithin{equation}{section}
\begin{document}
\title{\sffamily There is no variational characterization of the
cycles\\ in the method of periodic projections}
\author{J.-B. Baillon,$^1$ P. L. Combettes,$^{2}$ and 
R. Cominetti$^3$
\\[5mm]
\small
$\!^1$Universit\'e Paris 1 Panth\'eon-Sorbonne\\
\small SAMM -- EA 4543\\
\small 75013 Paris, France 
(Jean-Bernard.Baillon@univ-paris1.fr)\\[4mm]
\small $\!^2$UPMC Universit\'e Paris 06\\
\small Laboratoire Jacques-Louis Lions -- UMR 7598\\
\small 75005 Paris, France (plc@math.jussieu.fr)
\\[5mm]
\small
\small $\!^3$Universidad de Chile\\
\small Departamento de Ingenier\'{\i}a Industrial\\
\small Santiago, Chile (rccc@dii.uchile.cl)
}
\date{~}
\maketitle

\vskip 8mm

\begin{abstract} 
\noindent
The method of periodic projections consists in iterating projections 
onto $m$ closed convex subsets of a Hilbert space according to a
periodic sweeping strategy. In the presence of $m\geq 3$ sets, a 
long-standing question going back to the 1960s is whether the 
limit cycles obtained by such a process can be characterized as the 
minimizers of a certain functional. In this paper we
answer this question in the negative. Projection algorithms that 
minimize smooth convex functions over a product of convex sets are 
also discussed.
\end{abstract} 

\newpage

\section{Introduction}

Throughout this paper $\HH$ is a real Hilbert space with scalar
product $\scal{\cdot}{\cdot}$ and associated norm $\|\cdot\|$.
Let $C_1$ and $C_2$ be closed vector subspaces of $\HH$, and let 
$P_1$ and $P_2$ be their respective projection operators. The 
method of alternating projections for finding the projection of a 
point $x_0\in\HH$ onto $C_1\cap C_2$ is governed by the iterations 
\begin{equation}
\label{e:pocs1}
(\forall n\in\NN)\quad 
\begin{array}{l}
\left\lfloor
\begin{array}{l}
x_{2n+1}=P_2x_{2n}\\
x_{2n+2}= P_1x_{2n+1}.
\end{array}
\right.\\[2mm]
\end{array}
\end{equation}
This basic process, which can be traced back to Schwarz' alternating
method in partial differential equations \cite{Schw70}, has found 
many applications in mathematics and in the applied sciences; see 
\cite{Deut92} and the references therein. The strong convergence of
the sequence $(x_{n})_{n\in\NN}$ produced by \eqref{e:pocs1} to the
projection of $x_0$ onto $C_1\cap C_2$ was established by von 
Neumann in 1933 \cite{Vonn49}. The extension of \eqref{e:pocs1} to 
the case when $C_1$ and $C_2$ are general nonempty closed convex 
sets was considered in \cite{Che59a,Lev66b}. Thus, it was shown in 
\cite{Che59a} that, if $C_1$ is compact,
the sequences $(x_{2n})_{n\in\NN}$ and $(x_{2n+1})_{n\in\NN}$ 
produced by \eqref{e:pocs1} converge strongly to points 
$\overline{y}_1$ and $\overline{y}_2$, respectively, that constitute
a cycle, i.e.,
\begin{equation}
\label{e:2010-12-21b}
\overline{y}_1=P_1\overline{y}_2
\quad\text{and}\quad 
\overline{y}_2=P_2\overline{y}_1,
\end{equation}
or, equivalently, that solve the variational problem 
(see Figure~\ref{fig:1})
\begin{equation} 
\label{e:best2} 
\minimize{y_1\in C_1,\,y_2\in C_2}{\|y_1-y_2\|}.
\end{equation} 
Furthermore, it was shown in \cite{Lev66b} that, if $C_1$ is merely
bounded, the same conclusion holds provided strong convergence is 
replaced by weak convergence. As was proved only recently 
\cite{Hund04}, strong convergence can fail. 

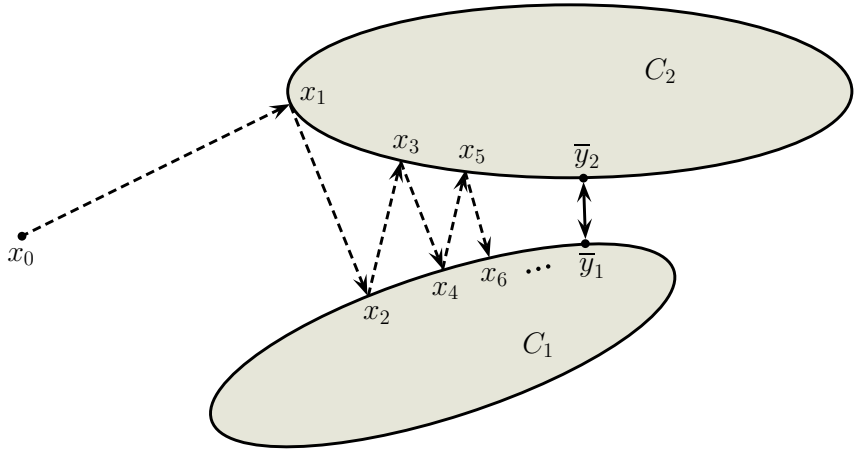
\begin{figure}[h!tb]
\begin{center}
\scalebox{0.65}{
\begin{pspicture}(-1,-4.85)(16.24,6.1)
\definecolor{color96b}{rgb}{0.90,0.90,0.85}
\rput{18.0}(-0.18953674,-2.8316424)%
{\psellipse[linewidth=0.06,dimen=outer,fillstyle=solid,%
fillcolor=color96b](8.844375,-2.0141652)(5.0,1.5)}
\psellipse[linewidth=0.06,dimen=outer,fillstyle=solid,%
fillcolor=color96b](11.444375,3.185835)(5.8,1.8)
\psline[linewidth=0.06cm,linestyle=dashed,linecolor=black,%
arrowsize=0.18cm 2.0,arrowlength=1.4,arrowinset=0.4]{->}%
(0.204375,0.20583488)(5.724375,2.925835)
\psline[linewidth=0.06cm,linestyle=dashed,linecolor=black,%
arrowsize=0.18cm 2.0,arrowlength=1.4,arrowinset=0.4]{->}%
(5.7243,2.9058)(7.304375,-0.994165)
\psline[linewidth=0.06cm,linestyle=dashed,linecolor=black,%
arrowsize=0.18cm 2.0,arrowlength=1.4,%
arrowinset=0.4]{->}(7.324375,-0.9941651)(7.984375,1.7458348)
\psline[linewidth=0.06cm,linestyle=dashed,linecolor=black,%
arrowsize=0.18cm 2.0,arrowlength=1.4,%
arrowinset=0.4]{->}(7.984375,1.7458348)(8.844375,-0.45416513)
\psline[linewidth=0.06cm,linestyle=dashed,linecolor=black,%
arrowsize=0.18cm 2.0,arrowlength=1.4,%
arrowinset=0.4]{->}(8.844375,-0.45416513)(9.3,1.55)
\psline[linewidth=0.06cm,linestyle=dashed,linecolor=black,%
arrowsize=0.18cm 2.0,arrowlength=1.4,%
arrowinset=0.4]{->}(9.3,1.55)(9.8,-0.22)
\psline[linewidth=0.06cm,arrowsize=0.18cm 2.0,arrowlength=1.4,%
arrowinset=0.4]{<->}(11.764375,0.16)(11.724375,1.33)
\psdots[dotsize=0.18](0.24,0.22)
\rput(10.8,-2.0){\LARGE $C_1$}
\rput(13.3,3.6){\LARGE $C_2$}
\rput(0.22,-0.20){\LARGE $x_0$}
\rput(7.5,-1.4){\LARGE $x_2$}
\rput(6.2,3.1){\LARGE $x_1$}
\rput(8.1,2.1){\LARGE $x_3$}
\rput(9.45,1.9){\LARGE $x_5$}
\rput(8.9,-0.89){\LARGE $x_4$}
\rput(9.9,-0.62){\LARGE $x_6$}
\rput(11.9,-0.37){\LARGE $\overline{y}_1$}
\rput(11.8,1.87){\LARGE $\overline{y}_2$}
\psdots[dotsize=0.09](10.6,-0.50)
\psdots[dotsize=0.09](10.8,-0.45)
\psdots[dotsize=0.09](11.0,-0.40)
\psdots[dotsize=0.18](11.764375,0.06583487)
\psdots[dotsize=0.18](11.724375,1.4058349)
\end{pspicture} 
}
\caption{In the case of $m=2$ sets, the method of alternating
projections produces a cycle $(\overline{y}_1,\overline{y}_2)$ that 
achieves the minimal distance between the two sets.}
\label{fig:1}
\end{center}
\end{figure}

Extending the above results to $m\geq 3$ nonempty closed convex 
subsets $(C_i)_{1\leq i\leq m}$ of $\HH$ poses interesting 
challenges. For 
instance, there are many strategies for scheduling the order in
which the sets are projected onto. The simplest one corresponds to a
periodic activation of the sets, say
\begin{equation}
\label{e:pocs2}
(\forall n\in\NN)\quad 
\begin{array}{l}
\left\lfloor
\begin{array}{ll}
x_{mn+1}&=P_mx_{mn}\\
x_{mn+2}&=P_{m-1}x_{mn+1}\\
&\;\vdots\\
x_{mn+m}&= P_1x_{mn+m-1},
\end{array}
\right.\\[2mm]
\end{array}
\end{equation}
where $(P_i)_{1\leq i\leq m}$ denote the respective projection
operators onto the sets $(C_i)_{1\leq i\leq m}$.
In the case of closed vector subspaces, it was shown in 1962 that 
the sequence $(x_n)_{n\in\NN}$ thus generated converges strongly to
the projection of $x_0$ onto $\bigcap_{i=1}^mC_i$ \cite{Halp62}. 
This provides a precise extension of the von Neumann result, 
which corresponds to $m=2$.
Interestingly, however, for nonperiodic sweeping strategies with
closed vector subspaces, only weak convergence has been established
in general \cite{Amem65} and, since 1965, it has remained an open
problem whether strong convergence holds (see \cite{Bail99} for 
the state-of-the-art on this conjecture). 
Another long-standing open problem is the
one that we address in this paper and which concerns the asymptotic 
behavior of the periodic projection algorithm \eqref{e:pocs2} for
general closed convex sets. It was shown in 1967 \cite{Gubi67} 
(see also \cite[Section~7]{Opti04}, \cite{Erem08}, and 
\cite[Th\'eor\`eme~5.5.2]{Mart72} for extensions of this result
to more general operators) that, if one of the sets is bounded, 
the sequences
$(x_{mn})_{n\in\NN}$, $(x_{mn+1})_{n\in\NN}$, 
\ldots, $(x_{mn+m-1})_{n\in\NN}$ converge 
weakly to points $\overline{y}_1$, $\overline{y}_m$, \ldots, 
$\overline{y}_2$, respectively, that constitute a cycle, i.e. 
(see Figure~\ref{fig:2}),
\begin{equation}
\label{e:2010-12-21c}
\overline{y}_1=P_1\overline{y}_2,\;\ldots,\;
\overline{y}_{m-1}=P_{m-1}\overline{y}_{m},\;
\overline{y}_m=P_m\overline{y}_{1}.
\end{equation}
However, it remains an open question whether, as in the case of 
$m=2$ sets, the cycles can be characterized as the solutions to a 
variational problem. We formally formulate this problem as follows.

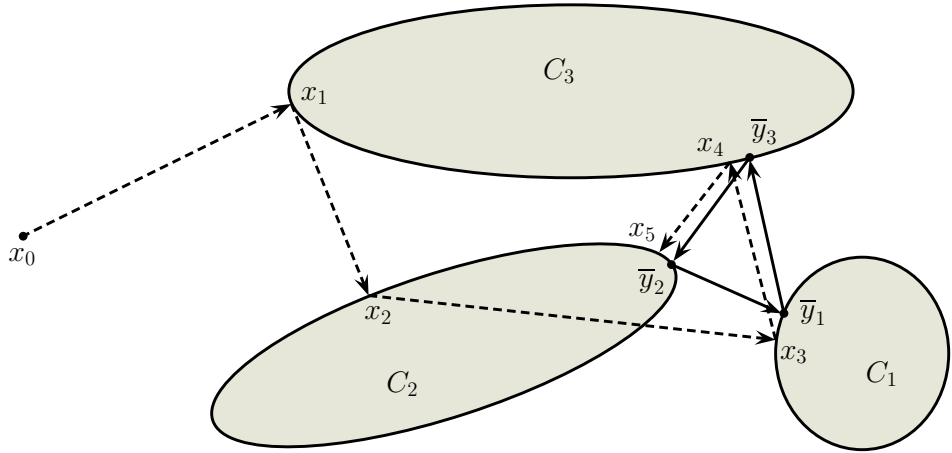
\begin{figure}[h!tb]
\begin{center}
\scalebox{0.65}{
\begin{pspicture}(-1,-4.85)(18.24,6.1)
\definecolor{color96b}{rgb}{0.90,0.90,0.85}
\rput{18.0}(-0.18953674,-2.8316424)%
{\psellipse[linewidth=0.06,dimen=outer,fillstyle=solid,%
fillcolor=color96b](8.844375,-2.0141652)(5.0,1.5)}
\psellipse[linewidth=0.06,dimen=outer,fillstyle=solid,%
fillcolor=color96b](11.444375,3.185835)(5.8,1.8)
\psellipse[linewidth=0.06,dimen=outer,fillstyle=solid,%
fillcolor=color96b](17.4,-2.18)(1.8,2.0)
\psline[linewidth=0.06cm,linestyle=dashed,linecolor=black,%
arrowsize=0.18cm 2.0,arrowlength=1.4,arrowinset=0.4]{->}%
(0.204375,0.20583488)(5.7243,2.926)
\psline[linewidth=0.06cm,linestyle=dashed,linecolor=black,%
arrowsize=0.18cm 2.0,arrowlength=1.4,arrowinset=0.4]{->}%
(5.7243,2.926)(7.304375,-0.994165)
\psline[linewidth=0.06cm,linestyle=dashed,linecolor=black,%
arrowsize=0.18cm 2.0,arrowlength=1.4,%
arrowinset=0.4]{->}(7.304375,-0.994165)(15.63,-1.9)
\psline[linewidth=0.06cm,linestyle=dashed,linecolor=black,%
arrowsize=0.18cm 2.0,arrowlength=1.4,%
arrowinset=0.4]{->}(15.63,-1.9)(14.70,1.74)
\psline[linewidth=0.06cm,linestyle=dashed,linecolor=black,%
arrowsize=0.18cm 2.0,arrowlength=1.4,%
arrowinset=0.4]{->}(14.70,1.74)(13.25,-0.11)
\psline[linewidth=0.06cm,arrowsize=0.18cm 2.0,arrowlength=1.4,%
arrowinset=0.4]{->}(15.1,1.83)(13.54,-0.29)
\psline[linewidth=0.06cm,arrowsize=0.18cm 2.0,arrowlength=1.4,%
arrowinset=0.4]{->}(13.5,-0.36)(15.75,-1.36)
\psline[linewidth=0.06cm,arrowsize=0.18cm 2.0,arrowlength=1.4,%
arrowinset=0.4]{->}(15.8,-1.36)(15.1,1.78)
\psdots[dotsize=0.18](0.24,0.22)
\rput(8.0,-2.8){\LARGE $C_2$}
\rput(11.2,3.6){\LARGE $C_3$}
\rput(17.8,-2.6){\LARGE $C_1$}
\rput(0.22,-0.20){\LARGE $x_0$}
\rput(7.5,-1.4){\LARGE $x_2$}
\rput(6.2,3.1){\LARGE $x_1$}
\rput(16.0,-2.2){\LARGE $x_3$}
\rput(14.3,2.04){\LARGE $x_4$}
\rput(12.9,0.3){\LARGE $x_5$}
\rput(13.1,-0.7){\LARGE $\overline{y}_2$}
\rput(15.4,2.37){\LARGE $\overline{y}_3$}
\rput(16.4,-1.33){\LARGE $\overline{y}_1$}
\psdots[dotsize=0.18](15.1,1.83)
\psdots[dotsize=0.18](13.5,-0.36)
\psdots[dotsize=0.18](15.8,-1.36)
\end{pspicture} 
}
\caption{Example with $m=3$ sets: the method of periodic
projections initialized at $x_0$ produces the cycle 
$(\overline{y}_1,\overline{y}_2,\overline{y}_3)$.}
\label{fig:2}
\end{center}
\end{figure}

\begin{definition}
\label{d:cycle}
Let $m$ be an integer at least equal to $2$ and let 
$(C_1,\ldots,C_m)$ be an ordered family of nonempty 
closed convex subsets of $\HH$ with associated projection operators 
$(P_1,\ldots,P_m)$. The set of cycles associated with 
$(C_1,\ldots,C_m)$ is
\begin{equation}
\label{e:2010-12-22a}
\cyc(C_1,\ldots,C_m)=
\menge{(\overline{y}_1,\ldots,\overline{y}_m)\in\HH^m}{
\overline{y}_1=P_1\overline{y}_2,\;\ldots,\;
\overline{y}_{m-1}=P_{m-1}\overline{y}_{m},\;
\overline{y}_m=P_m\overline{y}_{1}}.
\end{equation}
\end{definition}

\begin{question}
\label{prob:1}
Let $m$ be an integer at least equal to $3$. Does there exist a 
function $\Phi\colon\HH^m\to\RR$ such that, for every ordered 
family of nonempty closed convex subsets 
$(C_1,\ldots,C_m)$ of $\HH$, $\cyc(C_1,\ldots,C_m)$ is the set 
of solutions to the variational problem 
\begin{equation} 
\label{e:best3} 
\minimize{y_1\in C_1,\ldots,\,y_m\in C_m}{\Phi(y_1,\ldots,y_m)}\:?
\end{equation} 
\end{question}

Let us note that the motivations behind Question~\ref{prob:1} are 
not purely theoretical but also quite practical. Indeed, the 
variational properties of the cycles when $m=2$ have led to 
fruitful applications, e.g., \cite{Gold85,Merc80,Noba95,Pesq96}. 
Since the method of periodic projections \eqref{e:pocs2} is used 
in scenarios involving $m\geq 3$ possibly nonintersecting sets 
\cite{Proc93}, it is therefore important to understand the 
properties of its limit cycles and, in particular, whether they 
are optimal in some sense. Since the seminal work \cite{Gubi67}
in 1967 that first established the existence of cycles, little 
progress has been made towards this goal beyond the observation that
simple candidates such as 
$\Phi\colon(y_1,\ldots,y_m)\mapsto\|y_1-y_2\|+\cdots
+\|y_{m-1}-y_m\|+\|y_m-y_1\|$ fail 
\cite{Baus94,Baus97,Sign94,Kosm87}.
The main result of this paper is that the answer to
Question~\ref{prob:1} is actually negative. This result will be
established in Section~\ref{sec:2}. Finally, in 
Section~\ref{sec:3}, projection algorithms that are
pertinent to extensions of \eqref{e:best2} to $m\geq 3$ sets 
will be discussed.

\section{A negative answer to Question~\ref{prob:1}}
\label{sec:2}

We denote by $S(x;\rho)$ the sphere of center $x\in\HH$ and radius
$\rho\in\RP$, and by $P_C$ the projection operator onto a nonempty 
closed convex set $C\subset\HH$; in particular, $P_C 0$ is the
element of minimal norm in $C$.

Our main result hinges on the following variational property, 
which is of interest in its own right.

\begin{figure}[h!tb]
\begin{center}
\scalebox{0.60}{
\begin{pspicture}(-1,-7.0)(13.0,7.1)
\pscircle[linewidth=0.05,dimen=outer](6.0,0.0){6.0}
\psline[linewidth=0.04cm](6.0,0.0)(10.47,-4.0)
\psarc[linewidth=0.04](6.0,0.0){2.2}{-42.0}{-12.0}
\psline[linewidth=0.04cm](6.0,0.0)(11.873,-1.228)
\psline[linewidth=0.04cm](6.0,0.0)(11.7,1.873)
\psline[linewidth=0.04cm](6.0,0.0)(10.0,4.472)
\psline[linewidth=0.04cm](6.0,0.0)(7.228,5.873)
\psline[linewidth=0.04cm](6.0,0.0)(4.127,5.700)
\psline[linewidth=0.04cm](10.47,-4.0)(11.07,-1.07)
\psline[linewidth=0.04cm](11.07,-1.07)(10.27,1.4)
\psline[linewidth=0.04cm](10.27,1.4)(8.59,2.91)
\psline[linewidth=0.04cm](8.59,2.91)(6.7,3.3)
\psline[linewidth=0.04cm](6.7,3.3)(5.1,2.77)
\rput(11.7,-4.3){\LARGE $y=x_{n,0}$}
\rput(5.00,1.77){\LARGE $x$}
\rput(4.4,2.9){\LARGE $x_{n,n}$}
\rput(8.3,3.4){\LARGE $x_{n,k}$}
\rput(10.6,2.0){\LARGE $x_{n,k-1}$}
\rput(8.6,-1.3){\LARGE $\alpha/n$}
\rput(-0.5,-4.3){\LARGE $V$}
\psdots[dotsize=0.18](5.41,1.8)
\psdots[dotsize=0.18](10.47,-4.0)
\psdots[dotsize=0.18](11.07,-1.07)
\psdots[dotsize=0.18](10.27,1.4)
\psdots[dotsize=0.18](8.59,2.91)
\psdots[dotsize=0.18](6.7,3.3)
\psdots[dotsize=0.18](5.1,2.77)
\end{pspicture}
}
\end{center}
\caption{A polygonal spiral from $y=x_{n,0}$ to $x_{n,n}$ in $V$.}
\label{fig:3}
\end{figure}
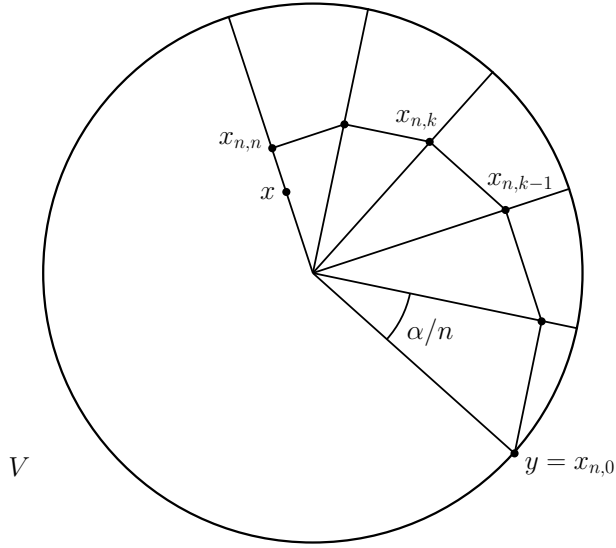

\begin{theorem}
\label{t:03} 
Suppose that $\dim\HH\geq 2$ and let $\varphi\colon\HH\to\RR$ be 
such that its infimum on every nonempty closed convex set 
$C\subset\HH$ is attained at $P_C 0$. Then the following hold.
\begin{enumerate}
\item
\label{t:03i} 
$\varphi$ is radially increasing, i.e.,
\begin{equation}
\label{e:mono}
(\forall x\in\HH)(\forall y\in\HH)\quad
\|x\|<\|y\|\quad\Rightarrow\quad\varphi(x)\leq\varphi(y).
\end{equation}
\item
\label{t:03ii} 
Suppose that, for every nonempty closed convex set $C\subset\HH$, 
$P_C 0$ is the unique minimizer of $\varphi$ on $C$.
Then $\varphi$ is strictly radially increasing, i.e.,
\begin{equation}
\label{e:monos}
(\forall x\in\HH)(\forall y\in\HH)\quad
\|x\|<\|y\|\quad\Rightarrow\quad\varphi(x)<\varphi(y).
\end{equation}
\item
\label{t:03iii} 
Except for at most countably many values of $\rho\in\RP$,
$\varphi$ is constant on $S(0;\rho)$.
\end{enumerate}
\end{theorem}
\proof 
\ref{t:03i}:
Let us fix $x$ and $y$ in $\HH$ such that $\|x\|<\|y\|$. 
If $x=0$, property \eqref{e:mono} amounts to the fact that $0$ is 
a global minimizer of $\varphi$, which follows from the assumption 
with $C=\HH$. We now suppose that $x\neq 0$. Let $V$ be a 
2-dimensional vector subspace of $\HH$ containing $x$ and $y$,
and let $\alpha\in[0,\pi]$ be the angle between $x$ and $y$. 
For every integer $n\geq 3$, consider a polygonal spiral built as 
follows: set $x_{n,0}=y$ and for $k=1,\ldots,n$ define
$x_{n,k}=P_{R_{n,k}} x_{n,k-1}$,
where $(R_{n,k})_{1\leq k\leq n}$ are $n$ angularly equispaced rays
in $V$ between the 
rays $\RP y$ and $\RP x=R_{n,n}$  (see Figure~\ref{fig:3}). 
Clearly, for the segment $C=[x_{n,k-1},x_{n,k}]$, we have
$P_C 0=x_{n,k}$, so that the assumption on $\varphi$ gives
$\varphi(x_{n,k})\leq\varphi(x_{n,k-1})$,
and therefore $\varphi(x_{n,n})\leq\varphi(x_{n,0})=\varphi(y)$. 
On the other hand, $x_{n,n}$ and $x$ are collinear with 
$\|x_{n,n}\|=\|y\|(\cos(\alpha/n))^n$ so that for $n$ large
enough we
have $\|x_{n,n}\|>\|x\|$ and, therefore, the segment
$C=[x,x_{n,n}]$ satisfies $P_C 0=x$, from which we get
$\varphi(x)\leq\varphi(x_{n,n})\leq\varphi(y)$
as claimed. 

\ref{t:03ii}:
If the minimizer of $\varphi$ on every nonempty closed convex set 
$C\subset\HH$ is unique, then all the inequalities above are strict.

\ref{t:03iii}:
Set $g\colon\RP\to\RR\colon\rho\mapsto\inf\varphi(S(0;\rho))$ 
and $h\colon\RP\to\RR\colon\rho\mapsto\sup\varphi(S(0;\rho))$.
It follows from \eqref{e:mono} that
\begin{equation}
(\forall\rho\in\RP)(\forall\rho'\in\RP)\quad
\rho<\rho'\quad\Rightarrow\quad g(\rho)\leq h(\rho)\leq
g(\rho')\leq h(\rho').
\end{equation}
Hence, $g$ and $h$ are increasing and therefore, by Froda's theorem
\cite[Theorem~4.30]{Rudi76}, the set of points at which they are
discontinuous is at most countable. 
Since $g$ and $h$ coincide at every point of continuity,
we conclude that, except for at most countably many $\rho\in\RP$,
$\varphi$ is constant on $S(0;\rho)$.
\endproof

As a straightforward consequence, we get the following.

\begin{corollary} 
Suppose that $\dim\HH\geq 2$ and let $\varphi\colon\HH\to\RR$ be 
such that its infimum on every nonempty closed convex set 
$C\subset\HH$ is attained at $P_C 0$. 
If $\varphi$ is either lower or upper semicontinuous, then
$\varphi=\theta\circ\|\cdot\|$, where 
$\theta\colon\RP\to\RR$ is increasing.
Furthermore, if $P_C 0$ is the unique minimizer of $\varphi$ on 
every nonempty closed convex set $C\subset\HH$, then $\theta$ is 
strictly increasing. 
\end{corollary}

Using Theorem \ref{t:03} we can provide the following
answer to Question \ref{prob:1}.

\begin{theorem} 
\label{t:4}
Suppose that $\dim\HH\geq 2$ and let $m$ be an integer at least 
equal to $3$. There exists no function $\Phi\colon\HH^m\to\RR$ 
such that, for every ordered family of nonempty closed convex 
subsets $(C_1,\ldots,C_m)$ of $\HH$, $\cyc(C_1,\ldots,C_m)$ is 
the set of solutions to the variational problem 
\begin{equation} 
\label{e:best4} 
\minimize{y_1\in C_1,\ldots,\,y_m\in C_m}{\Phi(y_1,\ldots,y_m)}.
\end{equation} 
\end{theorem}
\proof 
Suppose that $\Phi$ exists and set $(\forall i\in\{1,\ldots,m-2\})$
$C_i=\{0\}$. Moreover, take $z\in\HH$ and set $C_{m-1}=\{z\}$.
Then, for every nonempty closed convex set $C_m\subset\HH$ we have
\begin{equation}
\mathop{\rm Argmin}_{y_1\in C_1,\ldots,y_m\in C_m} 
\Phi(y_1,\ldots,y_m)=
\cyc(C_1,\ldots,C_m)=
\{(0,\ldots,0,z,P_{C_m} 0)\}.
\end{equation}
Hence, Theorem \ref{t:03} implies that, except for at most countably 
many values of $\rho\in\RP$, the function $\Phi(0,\ldots,0,z,\cdot)$ 
is constant on $S(0;\rho)$.

Now suppose that $z\in S(0;1)$ and take $\rho\in\left]1,\pinf\right[$ 
so that $\Phi(0,\ldots,0,z,\cdot)$ and $\Phi(0,\ldots,0,-z,\cdot)$
are constant on $S(0;\rho)$. Clearly, 
\begin{equation}
\cyc\big(\{0\},\ldots,\{0\},[-z,z],\{\rho z\}\big)
=\{(0,\ldots,0,z,\rho z)\}
\end{equation}
and
\begin{equation}
\cyc\big(\{0\},\ldots,\{0\},[-z,z],\{-\rho z\}\big)
=\{(0,\ldots,0,-z,-\rho z)\},
\end{equation}
so that 
\begin{align}
\Phi(0,\ldots,0,z,\rho z)
&<\Phi(0,\ldots,0,-z,\rho z)\nonumber\\
&=\Phi(0,\ldots,0,-z,-\rho z)\nonumber\\
&<\Phi(0,\ldots,0,z,-\rho z)\nonumber\\
&=\Phi(0,\ldots,0,z,\rho z)
\end{align}
where the inequalities come from the fact that the minima of $\Phi$
characterize the cycles,
while the equalities follow from the constancy of the functions 
on $S(0;\rho)$. Since these strict inequalities are impossible
it follows that $\Phi$ cannot exist.
\endproof

\section{Related projection algorithms}
\label{sec:3}

We have shown that the cycles produced by the method of cyclic 
projections \eqref{e:pocs2} are not characterized as the solutions
to a problem of the type \eqref{e:best3}, irrespective of the choice 
of the function $\Phi\colon\HH^m\to\RR$. 
Nonetheless, alternative projection methods can be
devised to solve variational problems over a product of closed 
convex sets. Here is an example.

\begin{theorem}
\label{t:2011-02-04}
For every $i\in I=\{1,\ldots,m\}$, let $(\HH_i,\|\cdot\|_i)$ be 
a real Hilbert space and let $C_i$ be a nonempty closed convex 
subset of $\HH_i$ with projection operator $P_i$. Let 
$\HHH$ be the Hilbert space obtained by endowing 
${\cart}_{i\in I}\HH_i$ with the norm 
$\boldsymbol{y}=(y_i)_{i\in I}\mapsto
\sqrt{\sum_{i\in I}\|y_i\|_i^2}$,
and let $\Phi\colon\HHH\to\RR$ be a 
differentiable convex function such that 
$\nabla\Phi\colon\HHH\to\HHH\colon{\boldsymbol y}\mapsto
(G_i{\boldsymbol y})_{i\in I}$
is $1/\beta$-lipschitzian for some $\beta\in\RPP$ and such that
the problem
\begin{equation} 
\label{e:best24} 
\minimize{y_1\in C_1,\ldots,\,y_m\in C_m}{\Phi(y_1,\ldots,y_m)}
\end{equation}
admits at least one solution. 
Let $\gamma\in\left]0,2\beta\right[$, 
set $\delta=\min\{1,\beta/\gamma\}+1/2$, let 
$(\lambda_n)_{n\in\NN}$ be a sequence in $[0,\delta]$ such that 
$\sum_{n\in\NN}\lambda_n(\delta-\lambda_n)=\pinf$, and let 
${\boldsymbol x}_0=(x_{i,0})_{i\in I}\in\HHH$. 
Set
\begin{equation}
\label{e:2011-02-04a}
(\forall n\in\NN)(\forall i\in I)\quad 
x_{i,n+1}=x_{i,n}+\lambda_{n}\big(P_i\big(x_{i,n}-\gamma
G_i{\boldsymbol x}_{n}\big)-x_{i,n}\big).
\end{equation}
Then, for every $i\in I$, $(x_{i,n})_{n\in\NN}$ converges weakly 
to a point $\overline{y}_i\in C_i$, and 
$(\overline{y}_i)_{i\in I}$ is a solution to \eqref{e:best24}.
\end{theorem}
\begin{proof}
Set ${\boldsymbol C}={\cart}_{i\in I}C_i$. Then ${\boldsymbol C}$ 
is a nonempty closed convex subset of $\HHH$ with projection 
operator 
$P_{\boldsymbol C}\colon{\boldsymbol x}\mapsto(P_ix_i)_{i\in I}$ 
\cite[Proposition~28.3]{Livre1}. Accordingly, we can rewrite
\eqref{e:2011-02-04a} as
\begin{equation}
\label{e:genna07-16}
(\forall n\in\NN)\quad 
{\boldsymbol x}_{n+1}={\boldsymbol x}_n+
\lambda_n\big(P_{\boldsymbol C}
\big({\boldsymbol x}_n-\gamma\nabla \Phi({\boldsymbol x}_n)\big)
-{\boldsymbol x}_n\big).
\end{equation}
It follows from \cite[Corollary~27.10]{Livre1} that 
$({\boldsymbol x}_n)_{n\in\NN}$ converges weakly to a 
minimizer $\overline{{\boldsymbol y}}$ of $\Phi$ over 
${\boldsymbol C}$, which concludes the proof.
\end{proof}

The projection algorithm described in the next result solves
an extension of \eqref{e:best2} to $m\geq 3$ sets.

\begin{corollary}
\label{c:2011-02-04}
Let $m$ be an integer at least equal to 3. For every 
$i\in I=\{1,\ldots,m\}$, let $C_i$ be a nonempty closed convex 
subset of $\HH$ with projection operator $P_i$, and let 
$x_{i,0}\in\HH$. 
Suppose that one of the sets in $(C_i)_{i\in I}$ is bounded
and set
\begin{equation}
\label{e:2011-02-04b}
(\forall n\in\NN)(\forall i\in I)\quad x_{i,n+1}=
P_i\left(\Frac{1}{m-1}\sum_{j\in I\smallsetminus\{i\}}
x_{j,n}\right).
\end{equation}
Then for every $i\in I$, $(x_{i,n})_{n\in\NN}$ converges weakly 
to a point $\overline{y}_i\in C_i$, and 
$(\overline{y}_i)_{i\in I}$ is a solution to the variational 
problem 
\begin{equation} 
\label{e:best8} 
\minimize{y_1\in C_1,\ldots,\,y_m\in C_m}
{\sum_{\substack{(i,j)\in I^2\\ i<j}}\|y_i-y_j\|^2}.
\end{equation}
Moreover, $\overline{y}=(1/m)\sum_{i\in I}\overline{y}_i$ is a
minimizer of the function 
$\varphi\colon\HH\to\RR\colon y\mapsto\sum_{i\in I}\|y-P_iy\|^2$.
\end{corollary}
\begin{proof}
We use the notation of Theorem~\ref{t:2011-02-04},
with $(\forall i\in I)$ $\HH_i=\HH$. 
Set $\beta=1-1/m$, $\gamma=1$, 
\begin{equation}
\Phi\colon\HHH\to\RR\colon(y_i)_{i\in I}\mapsto\frac{1}{2(m-1)}
\sum_{\substack{(i,j)\in I^2\\ i<j}}\|y_i-y_j\|^2,
\end{equation}
${\boldsymbol C}={\cart}_{i\in I}C_i$, and 
${\boldsymbol D}=\menge{(y,\ldots,y)\in\HHH}{y\in\HH}$.
Then \cite{Baus93,Sign94}
\begin{equation}
\label{e:2011-02-05a}
\Fix P_{\boldsymbol C}P_{\boldsymbol D}=
\operatorname{Argmin}\Phi\quad\text{and}\quad
\Fix P_{\boldsymbol D}P_{\boldsymbol C}=
\menge{(y,\ldots,y)}{y\in\operatorname{Argmin}\varphi}.
\end{equation}
Since one of the sets in $(C_i)_{i\in I}$ is bounded, 
$\operatorname{Argmin}\varphi\neq\emp$
\cite[Proposition~7]{Sign94}. 
Now let $y\in\operatorname{Argmin}\varphi$, and set
${\boldsymbol y}=(y,\ldots,y)$ and
${\boldsymbol x}=P_{\boldsymbol C}{\boldsymbol y}$.
Then \eqref{e:2011-02-05a} yields 
${\boldsymbol y}=P_{\boldsymbol D}P_{\boldsymbol C}{\boldsymbol y}$
and therefore
${\boldsymbol x}=P_{\boldsymbol C}
(P_{\boldsymbol D}P_{\boldsymbol C}{\boldsymbol y})
=P_{\boldsymbol C}P_{\boldsymbol D}{\boldsymbol x}$.
Hence ${\boldsymbol x}\in\operatorname{Argmin}\Phi$ and thus 
$\operatorname{Argmin}\Phi\neq\emp$. On the other hand, 
\eqref{e:best8} is a special case of \eqref{e:best24} 
and the gradient of $\Phi$ is the continuous linear operator 
\begin{equation}
\nabla\Phi\colon{\boldsymbol y}\mapsto\left(y_i-\Frac{1}{m-1}
\sum_{j\in I\smallsetminus\{i\}}y_j\right)_{i\in I}
\end{equation}
with norm $m/(m-1)=1/\beta$. Note that, since $m>2$, 
$2\beta>1=\gamma$.
Moreover, $\delta=\min\{1,\beta/\gamma\}+1/2>1$.
Thus, upon setting, for every $n\in\NN$, 
$\lambda_n\equiv 1\in\left]0,\delta\right[$ in 
\eqref{e:2011-02-04a}, we obtain \eqref{e:2011-02-04b}
and observe that $\sum_{n\in\NN}\lambda_n(\delta-\lambda_n)=\pinf$.
Altogether, the convergence result follows from 
Theorem~\ref{t:2011-02-04}. Finally, set
$\overline{\boldsymbol y}=(\overline{y}_1,\ldots,\overline{y}_m)$
and $\overline{\boldsymbol z}=P_{\boldsymbol D}
\overline{\boldsymbol y}$.
Then \eqref{e:2011-02-05a} yields
\begin{equation}
(\overline{y},\ldots,\overline{y})=
\overline{\boldsymbol z}=P_{\boldsymbol D}
\overline{\boldsymbol y}=P_{\boldsymbol D}(P_{\boldsymbol C}
P_{\boldsymbol D}\overline{\boldsymbol y})=
P_{\boldsymbol D}P_{\boldsymbol C}\overline{\boldsymbol z}
\end{equation}
and hence $\overline{y}\in\operatorname{Argmin}\varphi$.
\end{proof}

\begin{remark}
Alternative projection schemes can be derived from 
Theorem~\ref{t:2011-02-04}. For instance, 
Corollary~\ref{c:2011-02-04} remains valid if 
\eqref{e:2011-02-04b} is replaced by
\begin{equation}
\label{e:2011-02-06b}
(\forall n\in\NN)(\forall i\in I)\quad x_{i,n+1}=
P_i\left(\Frac{1}{m}\sum_{j\in I}x_{j,n}\right),
\end{equation}
which amounts to taking $\gamma=\beta$ instead of $\gamma=1$ in 
the above proof. We then recover a process investigated in
\cite{Baus94,Sign94,Pier85}.
\end{remark}


\begin{thebibliography}{99}

\small

\bibitem{Amem65} 
I. Amemiya and T. Ando, 
Convergence of random products of contractions in Hilbert space,
{\it Acta Sci. Math. (Szeged),}
vol. 26, pp. 239--244, 1965.

\bibitem{Bail99} 
J.-B. Baillon and R. E. Bruck,
On the random product of orthogonal projections in Hilbert space,
in: {\it Nonlinear Analysis and Convex Analysis}, pp. 126--133. 
World Scientific, River Edge, NJ, 1999. 

\bibitem{Baus93}
H. H. Bauschke and J. M. Borwein,
On the convergence of von Neumann's alternating projection 
algorithm for two sets, 
{\it Set-Valued Anal.,}
vol. 1, pp. 185--212, 1993.

\bibitem{Baus94}
H. H. Bauschke and J. M. Borwein,
Dykstra's alternating projection algorithm for two sets,
{\it J. Approx. Theory} 
vol. 79, pp. 418--443, 1994.

\bibitem{Baus97} 
H. H. Bauschke, J. M. Borwein, and A. S. Lewis, 
The method of cyclic projections for closed convex sets in Hilbert 
space, 
{\it Contemp. Math.,} 
vol. 204, pp. 1--38, 1997. 

\bibitem{Livre1} 
H. H. Bauschke and P. L. Combettes,
{\it Convex Analysis and Monotone Operator Theory in Hilbert 
Spaces.} Springer-Verlag, New York, 2011.

\bibitem{Che59a} 
W. Cheney and A. A. Goldstein, 
Proximity maps for convex sets,
{\it Proc. Amer. Math. Soc.,}
vol. 10, pp. 448--450, 1959.

\bibitem{Proc93} 
P. L. Combettes, 
The foundations of set theoretic estimation,
{\it  Proc. IEEE, }
vol. 81, pp. 182--208, 1993.

\bibitem{Sign94} 
P. L. Combettes, Inconsistent signal feasibility
problems: Least-squares solutions in a product space,
{\it IEEE Trans. Signal Process.,}
vol. 42, pp. 2955--2966, 1994.

\bibitem{Opti04}
P. L. Combettes, Solving monotone inclusions via compositions of 
nonexpansive averaged operators,
{\it Optimization,}
vol. 53, pp. 475--504, 2004.

\bibitem{Pier85} 
A. R. De Pierro and A. N. Iusem, A parallel projection
method for finding a common point of a family of convex sets,
{\it Pesquisa Oper.,}
vol. 5, pp. 1--20, 1985. 

\bibitem{Deut92} 
F. Deutsch, 
The method of alternating orthogonal projections, in:
{\it Approximation Theory, Spline Functions and Applications,}
(S. P. Singh, ed.), pp. 105--121. Kluwer, The Netherlands, 1992.

\bibitem{Erem08}
I. I. Eremin and L. D. Popov, Closed Fejer cycles for inconsistent 
systems of convex inequalities,
{\it Russian Math. (Iz. VUZ),}
vol. 52, pp. 8--16, 2008.

\bibitem{Gold85} 
M. Goldburg and R. J. Marks II, Signal 
synthesis in the presence of an inconsistent set of constraints,
{\it IEEE Trans. Circuits Syst.,}
vol. 32, pp. 647--663, 1985.

\bibitem{Gubi67} 
L. G. Gubin, B. T. Polyak, and E. V. Raik, 
The method of projections for finding the common point of 
convex sets, 
{\it Comput. Math. Math. Phys.,}
vol. 7, pp. 1--24, 1967.

\bibitem{Halp62} 
I. Halperin, The product of projection operators,
{\it Acta Sci. Math. (Szeged),}
vol. 23, pp. 96--99, 1962.

\bibitem{Hund04} 
H. S. Hundal,
An alternating projection that does not converge in norm,
{\it Nonlinear Anal.}, 
vol. 57, pp. 35--61, 2004.

\bibitem{Kosm87}
P. Kosmol, 
\"Uber die sukzessive Wahl des k\"urzesten Weges, in:
{\it \"Okonomie und Mathematik,} (O. Opitz and B. Rauhut, eds),
pp. 35--42. Springer-Verlag, Berlin, 1987.

\bibitem{Lev66b} 
E. S. Levitin and B. T. Polyak, 
Constrained minimization methods,
{\it Comput. Math. Math. Phys.,}
vol. 6, pp. 1--50, 1966.

\bibitem{Mart72} 
B. Martinet,
{\it Algorithmes pour la R\'esolution de Probl\`emes d'Optimisation 
et de Minimax.}
Th\`ese, Universit\'e de Grenoble, France, 1972.

\bibitem{Merc80} 
B. Mercier,  
{\it In\'equations Variationnelles de la M\'ecanique}
(Publications Math\'ematiques d'Orsay, no. 80.01).
Orsay, France, Universit\'e de Paris-XI, 1980. 

\bibitem{Vonn49} 
J. von Neumann, On rings of operators. Reduction theory,
{\it Ann. of Math.,}
vol. 50, pp. 401--485, 1949
(a reprint of lecture notes first distributed in 1933).

\bibitem{Noba95} 
R. A. Nobakht and M. R. Civanlar, 
Optimal pulse shape design for digital communication systems by 
projections onto convex sets,
{\it IEEE Trans. Communications,}
vol. 43, pp. 2874--2877, 1995.

\bibitem{Pesq96} 
J.-C. Pesquet and P. L. Combettes, 
Wavelet synthesis by alternating projections,
{\it IEEE Trans. Signal Process.,}
vol. 44, pp. 728--732, 1996.

\bibitem{Rudi76}
W. Rudin, 
{\it Principles of Mathematical Analysis,}
3rd ed. 
McGraw-Hill, New York, 1976.

\bibitem{Schw70} H. A. Schwarz, 
Grenz\"{u}bergang durch alternirendes Verfahren," 1870. 
Reprinted in {\it Gesammelte Mathematische Abhandlungen,} 
vol. 2, pp. 133--143. Springer-Verlag, Berlin, 1890.

\end{thebibliography}
\end{document}